\documentclass{amsart}
\usepackage{graphicx,,amssymb}
\usepackage{amsfonts}
\usepackage{mathrsfs}
\usepackage[all]{xy}
\usepackage{overpic}
\newtheorem{thm}{Theorem}[section]
\newtheorem{lemma}[thm]{Lemma}
\newtheorem{prop}[thm]{Proposition}
\newtheorem{cor}[thm]{Corollary}

\theoremstyle{definition}
\newtheorem{defi}[thm]{Definition}
\newtheorem{example}[thm]{Example}

\theoremstyle{remark}
\newtheorem{remark}[thm]{Remark}
\numberwithin{equation}{section}

\newcommand{\<}{\langle}
\renewcommand{\>}{\rangle}
\newcommand{\st}{\; | \;}             

\newcommand{\SSS}{\mathbb{S}}
\newcommand{\D}{\mathcal{D}}      
\newcommand{\Gammahat}{{\widehat{\Gamma}}}

\newcommand{\PQ}{P(\overrightarrow{Q})}

\newcommand{\Om}{\Omega}

\newcommand{\ep}{\epsilon}

\newcommand{\Z}{\mathbb{Z}}       

\newcommand{\K}{\mathcal{K}}
\newcommand{\N}{\mathbb{N}}
\newcommand{\gl}{\mathfrak{gl}}

\DeclareMathOperator{\Rep}{Rep}

\DeclareMathOperator{\Deg}{deg}
\DeclareMathOperator{\SRep}{SRep}
\DeclareMathOperator{\SVect}{\mathcal{SV}ect}
\begin{document}
\title{Quiver Representations in the Super-Category and Gabriel's Theorem for $A(n,m)$}
\author{J. Thind}
 \address{Department of Mathematics, University of Toronto, 
            Toronto, ON, Canada}
    \email{jthind@math.utoronto.ca}
\maketitle
\begin{abstract}
Gabriel's Theorem, and the work of Bernstein, Gelfand and Ponomarev established a connection between the theory of quiver representations and the theory of simple Lie algebras. Lie superalgebras have been studied from many perspectives, and many results about Lie algebras have analogues for Lie superalgebras. In this paper, the notion of a super-representation of a quiver is introduced, as well as the notion of reflection functors for odd roots. These ideas are then used to give a categorical construction of the root system $A(n,m)$ by establishing a version of Gabriel's Theorem and modifying the Bernstein, Gelfand, Ponomarev construction. This is then used to give a combinatorial description of the root system $A(n,m)$ where roots correspond to vertices of a canonically defined quiver $\Gammahat$. We also introduce a graded path algebra, whose category of graded modules is equivalent to the category of super-representations. This is used to define a graded preprojective algebra.
\end{abstract}

\section{Introduction}\label{s:intro}
Similar to the case of simple Lie algebras, the ``classical" simple Lie superalgebras have a classification by root system (see \cite{kac}). For Lie superalgebras the root system inherits a $\Z_{2}$ grading, and the odd roots exhibit different behaviour than the even roots. The notion of Dynkin diagram exists, but requires some extra care to define. Here the vertices are again simple roots, however they are coloured according to their parity. Unlike the classical case, the resulting diagram is dependent on the choice of simple roots. Whereas simple Lie algebras are classified by Dynkin diagrams, the same Lie superalgebra may have very different Dynkin diagrams, where not only does the colouring change, but the underlying graph does as well. 

However, the root system of type $A(n,m)$ is rather simple, and by ignoring grading, can be identified with a root system of type $A_{n+m-1}$, independent of the choice of simple roots. In particular, the underlying graph of the coloured Dynkin diagram is always of type $A_{n+m-1}$. Gabriel's Theorem, and the results of Bernstein, Gelfand and Ponomarev (BGP), give a categorical construction of the root system $A_{n+m-1}$ from the corresponding Dynkin graph $\Gamma$ via quiver representations.

Briefly, a quiver is an oriented graph, and a representation of a quiver is the choice of vector spaces at each vertex and linear maps corresponding to the edges. (For more details, see Section~\ref{s:quivers}.) 

Moduli spaces of quiver representations have provided geometric constructions of bases for representations of Lie algebras; the canonical bases of Lusztig, and the crystal bases of Kashiwara, via quiver varieties (see \cite{lusztig}, \cite{nak1}, \cite{nak2}, \cite{kas}, \cite{kash}). For Lie superalgebras, crystal bases have been shown to exist for $\gl (n,m)$ in \cite{bkk}. The author is currently working with I. Dimitrov on geometric constructions of crystal bases for $\gl(n,m)$. The motivation for this paper is to set up an appropriate framework for quiver representations in the super-category.

This paper introduces the category of super-representations of a quiver, as well as reflection functors for odd roots. In this category, indecomposable objects can be assigned a parity, and so we can use this to put a $\Z_{2}$ grading on the Grothendieck group of this category.

The main result (Theorem~\ref{t:main}) uses the notions of super-representations and reflection functors to extend Gabriel's Theorem and the BGP construction to the root system $A(n,m)$, and gives a combinatorial description of $A(n,m)$ via the Auslander-Reiten quiver similar to \cite{kt}. In particular, given a positive root $\alpha \in A(n,m)$, we construct an indecomposable super-representation $X_{\alpha}$, of dimension $\alpha$, with the same parity as $\alpha$. This gives a categorification, and combinatorial construction, of the root system $A(n,m)$ from any of its coloured Dynkin diagrams.

We begin with a review of the root system $A(n,m)$. After reviewing the relevant basics of quiver theory, we recall the construction of indecomposable representations by Bernstein, Gelfand and Ponomarev (see \cite{bgp}). We then introduce the category of super-representations of a quiver, and reflection functors for these categories. We then state and prove the Main Theorem. Finally, we introduce a graded version of the path algebra, show that the category of $\Z_{2}$-graded modules over this algebra is equivalent to the category of super-representations of the quiver, and define a graded preprojective algebra. The study of this preprojective algebra and its relationship to the representation theory of the corresponding Lie superalgebra is the subject of on-going research.

A slightly different notion of super-representation of a quiver appeared in \cite{leites}, however that paper contained no details, and to the best of our knowledge, no thorough treatment has appeared in the literature. 

The author would like to thank I. Dimitrov for many helpful discussions.

\begin{remark} This paper updates a previous version. Aside from correcting some errors, we have modified our definitions and constructions, and added some new material (graded path algebra). In the previous version of this paper, we also used a different definition of super-representation of a quiver. The earlier definition made certain notions easier (such as reflection functors), but was not compatible with the definition of modules over the graded path algebra. 
\end{remark}

\section{Root Systems of Type $A(n,m)$}\label{s:prelim}
In this section we briefly recall the basics of roots systems of type $A(n,m)$ following the presentation in \cite{kac}.

The root system of a Lie superalgebra has a splitting into even and odd parts coming from the $\Z_{2}$ grading on the corresponding Lie superalgebra; if we denote by $R$ the root system, then $R=R_{0} \sqcup R_{1}$, where $R_{0}, R_{1}$ are the even and odd parts respectively. This defines a parity function $p: R \to \Z_{2}$.

For type $A(n,m)$ the roots are expressed in terms of functionals $\ep_{1}, \ldots , \ep_{n}$, $\delta_{1}, \ldots , \delta_{m}$. If we write an element of $sl(n,m)$ as a block matrix $X=\left(\begin{array}{c|c} A & B \\\hline C & D \end{array}\right)$ with $A - n \times n$, $B - m\times n$, $C - n\times m$, $D - m \times m$, then $\ep_{i} (X) = A_{ii}$, and $\delta_{i} (X) = D_{ii}$. The roots and the $\Z_{2}$ grading are given by $R_{0} = \{ \ep_{i} - \ep_{j} , \delta_{i} - \delta_{j} | i\neq j \}$, $R_{1} = \{ \pm (\ep_{i} - \delta_{j} ) \}$.

As with the classical case, one can define a polarization into positive and negative parts, and define sets of simple roots. In the case of $A(n,m)$ the sets of simple roots, up to equivalence, are given by two sequences $S = \{s_{1} < s_{2} < \cdots \} , T=\{ t_{1} < t_{2} < \cdots \}$ and a choice of sign. For such choices, the corresponding set of simple roots is given by $\Pi_{S,T} = \pm \{ \ep_{1} -\ep_{2}, \ep_{2} - \ep_{3}, \ldots, \ep_{s_{1}} - \delta_{1}, \delta_{1} - \delta_{2}, \ldots, \delta_{t_{1}} - \ep_{s_{1}+1}, \ldots \}$. From $\Pi$ we construct a Dynkin diagram as in the classical case, except that we colour the vertices corresponding to the odd simple roots by $\bigotimes$. Denote the coloured Dynkin diagram by $\Gamma_{col}$. 

\begin{example}\label{e:ex1}
Consider the case of $A(2,2)$. The roots are $R_{0} = \{ \pm (\ep_{1} - \ep_{2}), \pm (\delta_{1} - \delta_{2}) \}$, $R_{1} = \{ \pm (\ep_{1} - \delta_{1}) , \pm (\ep_{1} -\delta_{2}) , \pm (\ep_{2} -\delta_{1} ), \pm (\ep_{2} - \delta_{2}) \}$, and we can take $\Pi = \{ \alpha_{1} = \ep_{1} - \ep_{2} , \alpha_{2} = \ep_{2} - \delta_{1} , \alpha_{3} = \delta_{1} - \delta_{2} \}$ as a set simple roots. This gives the following Dynkin diagram.

$$\xy
(0,0)*{\bigcirc}; (10,0)*{\bigotimes} **\dir{-};
(11.5,0)*{}; (20,0)*{\bigcirc} **\dir{-}; 
\endxy
$$

If instead, we chose the set of simple roots $\Pi^{\prime} = \{ \alpha_{1} = \ep_{1} - \delta_{1} , \alpha_{2} = \delta_{1} - \ep_{2} , \alpha_{3} = \ep_{2} - \delta_{2} \}$, we obtain the following Dynkin diagram.

$$
\xy
(0,0)*{\bigotimes}; (10,0)*{\bigotimes} **\dir{-};
(11.5,0)*{}; (20,0)*{\bigotimes} **\dir{-}; 
\endxy
$$

Note that in either case, the underlying uncoloured diagram is the same (a Dynkin diagram of type $A_{3}$).
\end{example} 

In the case of classical root systems, there is a group of automorphisms $W$ (Weyl group) acting simply transitively on the sets of simple roots by reflections. For Lie superalgebras this is not the case. In particular, the Weyl group does not act simply transitively. However, Serganova (see \cite{serg}) has developed a notion of ``Weyl groupoid" which plays the role of the Weyl group. The ``reflections" here are of two types: even reflections, which are honest reflections, and odd ``reflections" which describe how to change from a simple system containing the odd root $\alpha$ to the simple system containing $-\alpha$. In general, it is not easy to describe these odd reflections explicitly. However, for type $A(n,m)$ it is. By analogy with the classical case, we define a reflection for an odd simple root $\alpha_{i}$ by
$$s_{\alpha_{i} } (\alpha_{j}) = \begin{cases}
-\alpha_{i} &\text { if } i=j \\
\alpha_{i} + \alpha_{j}  &\text{ if } i-j  \text{ in } \Gamma \\
\alpha_{j} &\text{ otherwise.}
\end{cases}$$
This ``odd reflection" replaces the odd root $\alpha_{i}$ by its negative, and also changes the colouring of the neighbouring vertices in the coloured Dynkin diagram. (See \cite{serg} for details on odd reflections.) Note that an even reflection leaves the colouring of the Dynkin diagram fixed.

\begin{example}\label{e:ex2}
In the case of $A(2,2)$ considered in Example~\ref{e:ex1}, the simple system $\Pi^{\prime}$ was obtained from $\Pi$ by the odd reflection $s_{\alpha_{2}}$.
\end{example}

\subsection{Explicit Identification A(n,m) $\to$ A$_{ {\bf n+m-1}}$}
Express the roots of $A_{k}$ in terms of $e_{i}$ with $1\leq i \leq k+1$. Then the roots are $e_{i} - e_{j}$ with $i\neq j$. Define a map $A(n,m) \to A_{n+m-1}$ by $\epsilon_{i} \mapsto e_{i}$ and $\delta_{i} \mapsto e_{n+i}$. For a root $\alpha \in A(n,m)$ denote its image in $A_{n+m-1}$ by $\bar{\alpha}$.

\section{Quivers}\label{s:quivers}
In this section we review the basics of quiver theory (for more details see \cite{crawley-boevey}, \cite{bgp}, \cite{kt}). In particular, we recall the notion of quiver representations, Gabriel's Theorem, BGP reflection functors and the construction of indecomposable representations given in \cite{bgp}. We then review the results of \cite{kt}.

\subsection{Basics} 
A quiver $\overrightarrow{Q}$ is an oriented graph. The vertex set is denoted by $Q_{0}$ and the arrow set is denoted by $Q_{1}$. In what follows a quiver is obtained by orienting a graph $\Gamma$. In such a case the quiver is denoted by $\overrightarrow{Q} = (\Gamma, \Om)$ where $\Om$ is an orientation of the graph $\Gamma$. There are two functions $s,t : Q_{1} \to Q_{0}$, called ``source" and ``target" respectively, defined on an oriented edge $e : i\to j$ by $s(e) = i$ and $t(e) = j$. Let $n_{ij}$ denote the number of edges between $i$ and $j$ in $\Gamma$.

Fix a field $\mathbb{K}$. A representation of a quiver $\overrightarrow{Q}$ is a choice of vector space $X(i)$ for every vertex in $Q_{0}$ and linear map $x_{e} : X(i) \to X(j)$ for every edge $e: i \to j$. A morphism $\Phi : X \to Y$ of representations is a collection of linear maps $\phi_{i} : X(i) \to Y(i)$ such that the following diagram is commutative for every edge $e: i \to j$.
$$
\xymatrix{
X(i) \ar[r]^{x_{e}} \ar[d]^{\phi_{i}} & X(j) \ar[d]_{\phi_{j}} \\
Y(i) \ar[r]^{y_{e}} & Y(j) \\
}$$
Denote the Abelian category of representations of $\overrightarrow{Q}$ by $\Rep (\overrightarrow{Q})$. 

For any representation $X$, define its dimension vector by $\underline{\dim} (X) = ( \dim X(i) )_{i\in Q_{0}} \in \Z^{Q_{0}}$. 

\subsection{Gabriel's Theorem}
The connection between Lie theory and Quiver Theory began with the results of Gabriel in \cite{gab}, and the proof of Gabriel's Theorem given by Bernstein, Gelfand and Ponomarev in \cite{bgp}.

\begin{thm}\label{t:gabriel} \cite{gab}
\begin{enumerate}
\item A quiver $(\Gamma, \Om)$ has finitely many isomorphism classes of indecomposable representations if and only if the underlying graph $\Gamma$ is a Dynkin diagram of type $A,D,E$.
\item In this case, there is a bijection between isomorphism classes of indecomposable representations and positive roots of the corresponding root system, given by $ [X] \mapsto \underline{\dim} (X).$ Moreover, this bijection induces an isomorphism between $\K (\Rep (\Gamma, \Om))$ and the corresponding root lattice $\Z^{Q_{0}}$. (Here $\K$ denotes the Grothendieck group.)
\end{enumerate}
\end{thm}

\begin{remark} One can extend this to obtain the full root system by considering the 2-periodic derived category $\D^{b} (\Gamma , \Om) /T^{2}$, where $\D^{b}$ denotes the bounded derived category, and $T$ denotes the translation functor. (See \cite{px2} for more details.)
\end{remark}

\subsection{BGP Reflection Functors} 
Let $\overrightarrow{\Gamma} = (\Gamma , \Om)$ be a quiver and let $i \in \Gamma_{0}$ be a sink (or source) in the orientation $\Om$. Define a new orientation $s_{i} \Om$ of $\Gamma$ by reversing all arrows at $i$. Hence the sink becomes a source (and the source becomes a sink).

Let $i \in \Gamma$ be a sink (or source) in the orientation $\Om$ and consider the categories $\Rep (\Gamma , \Om)$ and $\Rep (\Gamma , s_{i} \Om)$. Despite the fact that Gabriel's Theorem establishes a bijection between the Grothendieck groups of $\Rep (\Gamma , \Om)$ for different choices of $\Om$, it is not the case that these categories are equivalent for different choices of $\Om$. 

However there are nice functors $S_{i}^{\pm} : \Rep (\Gamma , \Om) \to \Rep (\Gamma , s_{i} \Om)$ between them. (See \cite{bgp}.)  These functors are left (respectively right) exact. Although these functors are not equivalences, the corresponding derived functors $RS_{i}^{+}$ (respectively $LS_{i}^{-}$) provides an equivalence of triangulated categories $\D^{b} (\Gamma, \Om) \to \D^{b} (\Gamma , s_{i} \Om)$. A brief description of these functors is given for the reader's convenience.
 
Let $\overrightarrow{\Gamma} = (\Gamma, \Om)$ and let $i$ be a source for $\Om$. Define a functor $S_{i}^{+} : \Rep (\Gamma, \Om) \to \Rep (\Gamma , s_{i} \Om)$ by setting
$$ S_{i}^{+} X (j) = \begin{cases}
coker (X(i) \to \bigoplus_{i \to k} X(k)) &\text{if} \ \ i=j \\
X(j) &\text{otherwise.}
\end{cases}
$$ 
For an edge $e: j \to k$ in $\Om$ let $\overline{e}$ denote the corresponding edge in $s_{i} \Om$, the map $S_{i}^{+} (x_{\overline{e}})$ is given by 
$$S_{i}^{+} (x_{\overline{e}}) = \begin{cases}
x_{e} &\text{if } s(e) \neq i \\
q_{e} : X(t(e)) \to coker (x_{e} ) &\text{if } s(e) = i
\end{cases}
$$
where $q_{e}$ is the quotient map $X(t(e)) \to coker (x_{e})$.

Similarly for $i$ a sink, define $S_{i}^{-} : \Rep (\Gamma, \Om) \to \Rep (\Gamma , s_{i} \Om)$ by 
$$ S_{i}^{-} X (j) = \begin{cases}
ker(\bigoplus_{i\to k} X(k) \to X(i) ) &\text{if} \ \ i=j \\
X(j) &\text{otherwise.}
\end{cases}
$$ 
For an edge $e: k \to j$ in $\Om$ let $\overline{e}$ denote the corresponding edge in $s_{i} \Om$, the map $S_{i}^{-} (x_{\overline{e}})$ is given by 
$$S_{i}^{-} (x_{\overline{e}}) = \begin{cases}
x_{e} &\text{if } t(e) \neq i \\ 
(\iota_{k}, 0 ) : ker (x_{e}) \to \oplus X(s(e)) &\text{if } t(e) = i
\end{cases}
$$
where $\iota_{k}$ is the embedding of $ker (x_{e})$ into $X(s(e))$.

These are the well-known ``BGP reflection functors" (see \cite{bgp}). Note that the derived functors $RS_{i}^{+}$ and $LS_{i}^{-}$ are inverse to each other. The name reflection functor comes from the action of these functors on the Grothendieck group. In the setting of Gabriel's Theorem, the Grothendieck group of $\D^{b} (\overrightarrow{\Gamma}) /T^{2}$ is isomorphic to the root lattice, indecomposable objects correspond to the roots and the reflection functors act on the Grothendieck group as the corresponding simple reflections in the Weyl group of the associated root system.

In \cite{bgp}, Bernstein, Gelfand and Ponomarev used these reflection functors to give an alternate proof of Gabriel's Theorem from a more Lie theoretic point of view. We will modify this proof, so we shall briefly explain their construction here.

In the rest of this section we consider only the BGP reflection functors at $i \in \Gamma$ a sink, so we drop the superscript and write $S_{i}$ for the reflection functor to simplify notation.

Let $\Gamma$ be of type $A,D,E$ and $\Om$ an orientation of $\Gamma$. For each $i \in \Gamma$ denote by $\SSS_{i}$ the simple representation defined by $$\SSS_{i} (j) = \begin{cases}
\mathbb{K} &\text{ if } i=j \\
0 &\text{ otherwise.} \\
\end{cases}$$
It is well known that these are the only simple representations (see \cite{crawley-boevey} for details).
\begin{defi} Let $W$ denote the Weyl group. A reduced expression $w=s_{i_{1}} s_{i_{2}} \cdots s_{i_{m}}$ for a word $w\in W$ is called {\em adapted to $\Om$} if $i_{k}$ is a sink for the orientation $s_{i_{k-1}} \cdots s_{i_{2}} s_{i_{1}} \Om$. (In particular $i_{1}$ is a sink for $\Om$.)
\end{defi}

\begin{example}
For the $A_{3}$ quiver $( \Gamma, \Om )$ given by $\bigcirc \longleftarrow \bigcirc \longleftarrow \bigcirc$, the reduced expression $w_{0}= s_{1} s_{2} s_{3} s_{1} s_{2} s_{1}$ is adapted to $\Om$.
\end{example}

Let $w_{0} = s_{i_{1}} s_{i_{2}} \cdots s_{i_{l}}$ be a reduced expression for the longest element, which is adapted to $\Om$. Then we can enumerate the positive roots by setting $\gamma_{k} = s_{i_{1}} s_{i_{2}} \cdots s_{i_{k-1}} \alpha_{i_{k}}$, where $\alpha_{i_{k}}$ is the simple root corresponding to vertex $i_{k}$.

\begin{example}
For the case of $A_{3}$ considered in the previous example, we get the following enumeration of the positive roots: $\gamma_{1} = \alpha_{1}, \gamma_{2} = \alpha_{1} + \alpha_{2}, \gamma_{3} = \alpha_{1} + \alpha_{2} + \alpha_{3}, \gamma_{4} = \alpha_{2}, \gamma_{5} = \alpha_{2} + \alpha_{3}, \gamma_{6} = \alpha_{3}$.
\end{example}

Let $\alpha$ be a positive root, and consider the unique expression $\alpha =s_{i_{1}} s_{i_{2}} \cdots s_{i_{k-1}} \alpha_{i_{k}}$. Define an object $X_{\alpha} \in \Rep(\Gamma , \Om)$ by setting $X_{\alpha} = S_{i_{1}} S_{i_{2}} \cdots S_{i_{k-1}} (\SSS_{i_{k}})$, where $\SSS_{i_{k}} \in \Rep(\Gamma , S_{i_{k-1}} \cdots S_{i_{2}} S_{i_{1}} \Om)$ is the simple object corresponding to vertex $i_{k}$. Bernstein, Gelfand and Ponomarev showed that this representation is indecomposable, and that up to isomorphism, these are the only indecomposables. 

\begin{thm}\label{t:bgp} \cite{bgp}
Let $\Gamma$ be a Dynkin diagram of type $A,D,E$, and $\Om$ an orientation of $\Gamma$. Let $\alpha \in R$ be a positive root, and $X_{\alpha}$ the object defined above. Then $X_{\alpha}$ is indecomposable, and $\underline{\dim}(X_{\alpha}) = \alpha$. Moreover, any indecomposable object is isomorphic to $X_{\alpha}$ for some $\alpha$.
\end{thm}

\subsection{Auslander-Reiten Quiver and Height Functions}
Given an A,D,E Dynkin diagram $\Gamma$ with Coxeter number $h$ define a quiver $\Gammahat$ by 

\begin{equation}\label{e:Ihat}
\Gammahat = \{ (i,n) \subset \Gamma \times \Z_{2h} \st n+p(i) \equiv 0 \mod 2 \}
\end{equation}
where vertex $(i,n)$ is connected to vertex $(j,n+1)$ for $i-j$ in $\Gamma$. The quiver $\Gammahat$ is called the ``{\em periodic Auslander-Reiten quiver}," and can be realised as the Auslander-Reiten quiver of the category $D^{b}(\Gamma , \Om)/T^{2}$, where $\D^{b}$ denotes the bounded derived category and $T$ is the translation functor. 

\begin{example}
The quiver $\Gammahat$ for the graph $A_{3}$ is shown below. Recall that $\Gammahat$ is cyclic, so the arrows leaving the top level are the same as the incoming arrows at the bottom level.
$$
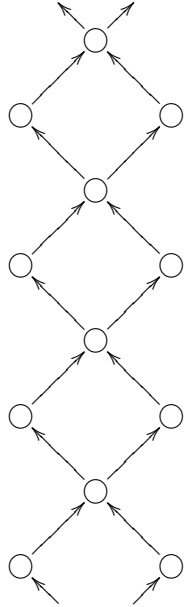
\xy
(0,0)*{\bigcirc}="A"; (10,10)*{\bigcirc}="B"; (0,20)*{\bigcirc}="C"; (10, 30)*{\bigcirc}="D"; (0,40)*{\bigcirc}="E"; (10,50)*{\bigcirc}="F"; (0,60)*{\bigcirc}="G"; (10,70)*{\bigcirc}="H"; (5,75)*{}="I"; (5,-5)*{}="P"; 
(20,0)*{\bigcirc}="J"; (20,20)*{\bigcirc}="K"; (20,40)*{\bigcirc}="L"; (20,60)*{\bigcirc}="M"; (15,75)*{}="N"; (15,-5)*{}="Q"
\ar "B"; "A" ;
\ar "C"; "B";
\ar "D"; "C";
\ar "E"; "D";
\ar "F"; "E";
\ar "G"; "F";
\ar "H"; "G";
\ar "I"; "H";
\ar "A"; "P";

\ar "B"; "J";
\ar "K"; "B";
\ar "D"; "K";
\ar "D" ; "L"
\ar "F"; "L";
\ar "M"; "F";
\ar "H"; "M";
\ar "N"; "H";
\ar "J"; "Q";
\endxy
$$
\end{example}

There is a map $\tau : \Gammahat \to \Gammahat$ defined by $\tau (i,n) = (i,n+2)$.

A {\em height function} $h: \Gamma \to \Z_{2h}$ is a function which satisfies the condition that $h(i) = h(j) \pm 1$ for $i-j$ in $\Gamma$. Any height function $h$ defines an orientation $\Om_{h}$ of $\Gamma$, given by $i \to j$  in $\Om_{h}$ if $h(j) = h(i) + 1$ and $i-j$ in $\Gamma$.

In \cite{kt}, a combinatorial construction of $R$ from $\Gamma$ is given, which realizes $R$ in terms of the quiver $\Gammahat$. For convenience, we summarize the results of \cite{kt} in Theorem~\ref{t:kt1}.

\begin{thm}\label{t:kt1}
Let $\Gamma$ be an $A,D,E$ Dynkin diagram, let $R$ denote the corresponding root system and $W$ its Weyl group. Fix $C\in W$ a Coxeter element. Then there is a canonical bijection $\Phi : R\to \Gammahat$ with the following properties:
\begin{enumerate}
\item It identifies the Coxeter element $C$ with the ``twist" $\tau: \Gammahat \to \Gammahat$. Hence the natural projection map $\pi : \Gammahat \to \Gamma$  given by $\pi (i,n) = i$ gives a bijection between orbits of the Coxeter element and vertices of the Dynkin diagram.
\item It gives a bijection between simple systems $\Pi$ compatible with $C$ and height functions $h: \Gamma \to \Gammahat$. 
\item For each height function $h$ there is an explicit description of the corresponding positive roots and negative roots as disjoint connected subquivers of $\Gammahat$, as well as a reduced expression for the longest element $w_{0}$ in the Weyl group.
\item There is a de-symmetrization $\< \cdot , \cdot \>$ of the inner product on $R$, which is analogous to the Euler form $\< \cdot , \cdot \>_{\overrightarrow{\Gamma}}$ in quiver theory. Moreover, under the bijection $\Phi$ this form admits an explicit description in terms of paths in $\Gammahat$.
\end{enumerate}
\end{thm}

\section{Super-Representations of a Quiver}\label{s:superrep}
In this section we introduce a notion of a ``super-representation" of a coloured quiver, and describe the corresponding category. We then define the analogue of the BGP reflection functors for the super-category.

For a $\Z_{2}$-coloured graph $\Gamma_{col}$ with underlying graph $\Gamma$, let  $p: \Gamma_{col} \to \Z_{2}$ be the parity function, denote by $\Gamma_{n} = \{\text{vertices coloured by }n \} = \{ i \in \Gamma \ | \ p(i) =n\} $, and write $\Gamma_{col} = \Gamma_{0} \cup \Gamma_{1}$. For a coloured graph, we can associate a quiver $(\Gamma_{col}, \Om)$ by orienting the edges.\\

\begin{defi}
A {\em super-representation} $X$ of a coloured quiver $( \Gamma_{col} ,\Om)$ is a choice of finite dimensional $\Z_{2}$-graded vector space $X (i) = X_{0}(i) \oplus X_{1}(i)$ for each $i \in \Gamma_{col}$ and homogeneous linear maps $x_{e} : X (s(e)) \to X (t(e))$ for every edge $e$, whose degree is given by the formula $\Deg (x_{e}) = p(s(e)) + p (t(e))$.
\end{defi}

\begin{defi} A {\em morphism} $\Phi: X \to Y$ between super-representations of $( \Gamma , \Om)$ is a collection of homogeneous linear maps $\phi_{i} : X(i) \to Y(i)$ such that for every edge $e: i \to j$ the following diagram is commutative:
$$
\xymatrix{
X(i) \ar[r]^{x_{e}} \ar[d]^{\phi_{i}} & X(j) \ar[d]_{\phi_{j}} \\
Y(i) \ar[r]^{y_{e}} & Y(j) \\
}$$
\end{defi}

Denote the abelian category of super-representations of $(\Gamma, \Om)$ by $\SRep (\Gamma , \Om)$. 

Let $P: \SVect \to \SVect$ be the parity change functor for super vector spaces. It is defined by $P(\mathbb{K} ^{n|m}) = \mathbb{K}^{m|n}$, and if $T: \mathbb{K}^{n|m} \to \mathbb{K}^{k|l}$ is given by a block matrix $T=\left(\begin{array}{c|c} A & B \\\hline C & D\end{array}\right)$, then $P(T) =\left(\begin{array}{c|c} D & C \\\hline B & A\end{array}\right) : \mathbb{K}^{m|n} \to \mathbb{K}^{l|k}$.

This induces a parity change functor on $\SRep (\Gamma , \Om)$, which we will also denote by $P : \SRep (\Gamma , \Om) \to \SRep (\Gamma , \Om)$. 

\begin{lemma}\label{l:parity}
For $X \in \SRep (\Gamma , \Om)$, $P(X) \simeq X$.
\end{lemma}
\begin{proof}
By definition of morphisms, we can construct an isomorphism $\Phi : X \to P(X)$, by setting $\phi_{i} : X(i) \to PX(i)$ to be the degree 1 linear map given by parity change $\mathbb{K} ^{n|m} \to \mathbb{K}^{m|n}$.
\end{proof}

\begin{remark} 
Note that in the category of super-vector spaces, morphisms are only the linear maps of degree 0, so parity change does not give an isomorphism. However, we allow linear maps of degree 1 as well, so that parity change does give an isomorphism. 
\end{remark}

\begin{prop}\label{p:indec}
Any object $X \in \SRep (\Gamma , \Om)$ is of the form $X= X^{\prime} \oplus X^{\prime \prime}$, where for every vertex $i$ the vector spaces $X^{\prime} (i), X^{\prime \prime} (i)$ are purely even or odd, and $p(X^{\prime} (i)) \neq p( X^{\prime \prime} (i))$.

In particular, if $X$ is indecomposable, then $X= X^{\prime} \oplus \mathbb{O}$ or $X = \mathbb{O} \oplus X^{\prime \prime}$.
\end{prop}

\begin{proof}
If $X(i) = \mathbb{K} ^{n_{i} | m_{i}}$, set 

$$X^{\prime}(i) =
\begin{cases}
\mathbb{K}^{n_{i} | 0} &\text{ if } p(i) = 0 \\
\mathbb{K}^{0|m_{i}} &\text{ if } p(i) = 1\\
\end{cases}
$$
and 
$$X^{\prime \prime} (i) =
\begin{cases}
\mathbb{K}^{ 0|m_{i}} &\text{ if } p(i) = 0 \\
\mathbb{K}^{n_{i}| 0} &\text{ if } p(i) = 1.\\
\end{cases}
$$
Define $x_{e} ^{\prime} = x_{e} |_{X^{\prime}}$ to be the restriction of $x_{e}$ to $X^{\prime}$. Simarly, define $x_{e} ^{\prime \prime} = x_{e} |_{X^{\prime \prime}}$ to be the restriction of $x_{e}$ to $X^{\prime \prime}$. Note that if $e: i \to j$, then $x_{e}$ is homogeneous, and $\Deg (x_{e}) = p(i) + p(j)$, so these maps are well-defined.
\end{proof}

\begin{remark} Note that this description allows us to define the ``parity" of an indecomposable object $X$ by setting $p(X) = \sum_{i} p(X(i))$. However, Lemma~\ref{l:parity} implies that this does not define a well-defined function on isomorphism classes of indecomposables, since $X \simeq P(X)$, but these objects have opposite parities. It turns out, however, that this definition of parity will still be useful. 
\end{remark}

Let $\Gamma_{col}$ be a coloured graph, with underlying graph $\Gamma$.

Define a functor $G : \Rep (\Gamma , \Om) \to \SRep (\Gamma_{col} , \Om)$ as follows:
$$GX (i) =
\begin{cases}
X(i) \oplus \mathbb{O} \text{ if } p(i) = 0 \\
\mathbb{O} \oplus X(i) \text{ if } p(i) = 1
\end{cases}
$$ 
and if $e : i \to j$
$$ Gx_{e} =
\begin{cases}
\left(\begin{array}{c|c} x_{e} & 0 \\\hline 0 & 0 \end{array}\right) \text{ if } p(i) = p(j) = 0 \\
\left(\begin{array}{c|c} 0 & 0 \\\hline 0 & x_{e} \end{array}\right) \text{ if } p(i) = p(j) = 1 \\
\left(\begin{array}{c|c} 0 & x_{e} \\\hline 0 & 0 \end{array}\right) \text{ if } p(i) = 0, p(j) = 1 \\
\left(\begin{array}{c|c} 0 & 0 \\\hline x_{e} & 0 \end{array}\right) \text{ if } p(i) = 1, p(j) = 0. \\
\end{cases}
$$
Let $F: \SRep  (\Gamma , \Om ) \to \Rep (\Gamma , \Om)$ be the forgetful functor given by ignoring the $\Z_{2}$ gradings of the vector spaces $X(i)$.

\begin{prop}
The functor $G$ is an equivalence of categories, with inverse given by $F$.
\end{prop}
\begin{proof}
It follows from Proposition~\ref{p:indec} and the definition of morphisms in $\SRep (\Gamma_{col} , \Om)$ that $F \circ G (X) = X$ and $G \circ F (X) \simeq X$.
\end{proof}

From this Proposition, we get Part 1 of Gabriel's Theorem for free (see Theorem~\ref{t:gabriel}), and a full description of indecomposable objects. By understanding the indecomposable objects, we can then use the coloured Dynkin diagram and the ideas from \cite{bgp} and \cite{kt} to get a categorical construction of the root system $A(m,n)$ similar to the description of $A,D,E$ root systems given by Gabriel's Theorem, and a combinatorial description of $A(n,m)$ similar to that in \cite{kt}. In particular, we extend the construction of \cite{bgp} by introducing reflection functors for $\SRep$, so that for each root $\alpha$ the indecomposable object $X_{\alpha}$ has the same parity as $\alpha$. This can then be used to colour the vertices in the Auslander-Reiten quiver $\Gammahat$ so that the bijection $R \to \Gammahat$ identifies even and odd roots in $A(n,m)$.

\subsection{Reflection Functors in the Super-Category}
Recall that for Lie superalgebras the construction of the coloured Dynkin diagram depends on the choice of simple roots. In general, not only the colouring, but the underlying graph depends on the choice of simple roots. However, for type $A(n,m)$ the underlying graph (of type $A_{n+m-1}$) is independent of the choice of simple roots. When we replace the simple system $\Pi$ by replacing the odd root $\alpha_{i}$ with $-\alpha_{i}$, the coloured Dynkin diagram is obtained by changing the colour of the vertices adjacent to $i$.

Let $( \Gamma_{col} , \Om)$ be a coloured quiver. If $i \in (\Gamma_{col} , \Om)$ is a sink or source, define a new quiver $(s_{i} \Gamma_{col} , s_{i} \Om)$ as follows:\\
\begin{itemize}
\item If $p(i) = 0$, then $s_{i} \Gamma_{col} = \Gamma_{col}$, and $s_{i} \Om$ is the orientation obtained by reversing all arrows at $i$.\\
\item If $p(i) = 1$, then $s_{i} \Gamma_{col}$ is the coloured graph obtained by changing the colour of all vertices adjacent to $i$, and $s_{i} \Om$ is the orientation obtained by reversing all arrows at $i$. (Compare with Example~\ref{e:ex1}.)\\
\end{itemize}
\begin{example}\label{e:reflections}
Consider the coloured quiver $(\Gamma_{col} , \Om) =  \bigcirc \longrightarrow \bigotimes \longleftarrow \bigcirc$, then\\ 
$(s_{2} \Gamma_{col} , s_{2} \Om) = \bigotimes \longleftarrow \bigotimes \longrightarrow \bigotimes$, while $( s_{1} \Gamma_{col} , s_{1} \Om) = \bigcirc \longleftarrow \bigotimes \longleftarrow \bigcirc$.

Here we colour vertex $i$ by $\bigcirc$ if $p(i) = 0$, and $\bigotimes$ if $p(i) = 1$.
\end{example}

We define reflection functors $\widetilde{S_{i}^{\pm}}$ as follows.

\begin{defi}\label{d:superrefl}
Let $\Gamma_{col}$ be a $\Z_{2}$-coloured graph and $\Om$ be an orientation of $\Gamma_{col}$. Let $d(i,j)$ denote the number of edges between $i$ and $j$ in $\Gamma_{col}$.\\

If $i \in (\Gamma_{col} , \Om)$ a sink, define functors $\widetilde{S_{i}^{\pm}} :\SRep (\Gamma , \Om) \to \SRep (s_{i} \Gamma , s_{i} \Om)$ as follows:

$$\widetilde{S_{i}^{\pm}} (X) (j) = \begin{cases}
P^{p(i)} \circ S_{i}^{\pm} X(j) &\text{ if } d(i,j) \leq 1 \\
X(j) &\text{ if } d(i,j) \geq 2.
\end{cases}
$$

For an edge $e$ in $\Om$ with let $\overline{e}$ denote the corresponding edge in $s_{i} \Om$. (If $s(e) \neq i$, or $t(e) \neq i$, then $\overline{e} = e$.) Then the map $\widetilde{S_{i}^{\pm}} (x_{\overline{e}})$ is given by :
$$\widetilde{S_{i}^{\pm}} (x_{\overline{e}}) = \begin{cases}
P^{p(i)} \circ S_{i}^{\pm} (x_{e}) &\text{ if } s(e) = i, \text{ or if } s(e) \text{ or } t(e) \text{ is adjacent to } i. \\
x_{e} &\text{ otherwise.}
\end{cases}
$$

where $S_{i}^{\pm}$ are defined in Section~\ref{s:quivers}, and $P$ is parity change.
\end{defi}

Note that if $p(i) = 0$, then $\widetilde{S_{i}^{\pm}} = S_{i}^{\pm}$.

\begin{example}
For $X =\stackrel{1|0}{\bigcirc} \longleftarrow \stackrel{1|0}{\bigcirc} \longleftarrow \stackrel{1|0}{\bigcirc} \longrightarrow \stackrel{0|0}{\bigotimes} \longleftarrow \stackrel{0|1}{\bigotimes}$, if we apply the reflection functor $\widetilde{S_{4}^{-}}$ we get $\widetilde{S_{4}^{-}} X = \stackrel{1|0}{\bigcirc} \longleftarrow \stackrel{1|0}{\bigcirc} \longleftarrow \stackrel{0|1}{\bigotimes} \longleftarrow \stackrel{1|1}{\bigotimes} \longrightarrow \stackrel{1|0}{\bigcirc}$. 

The non-zero maps in $X$ are the identity, and the non-zero maps in $\widetilde{S_{4}^{-}} X$ are either the identity or parity change, except the maps leaving vertex $4$, which are the obvious inclusions.
\end{example}

\section{Main Results}\label{s:main}
In this section we give our main results. In particular, we show that the construction of indecomposables given in \cite{bgp} can be extended to $A(n,m)$, using the reflection functors for the category of super-representations of the quiver $(\Gamma_{col} , \Om)$.

To begin, recall the identification $A(n,m) \to A_{n+m-1}$ given by forgetting the grading on $A(n,m)$. Explicitly, this takes $\epsilon_{i} \mapsto e_{i}$ for $1\leq i \leq n$, and $\delta_{i} \mapsto e_{n+i}$ (see Section~\ref{s:prelim}).  For a root $\alpha \in A(n,m)$ consider its image in $\bar{\alpha} \in A_{n+m-1}$. 

Choose a set of simple roots $\Pi$ in $A(n,m)$ and let $\Gamma_{col}$ denote the corresponding coloured Dynkin diagram. Choose an orientation $\Om$ for $\Gamma_{col}$, and consider the quiver $(\Gamma_{col} , \Om)$. Consider the longest element $w_{0} \in W$ for the system $A_{n+m-1}$, and let $w_{0} = s_{i_{1}} s_{i_{2}} \cdots s_{i_{l}}$ be a reduced expression for $w_{0}$ adapted to $\Om$. 

Let $\alpha \in A(n,m)$ be a positive root with respect to $\Pi$, then we have an expression $\bar{\alpha} = s_{i_{1}} s_{i_{2}} \cdots s_{i_{j-1}} \bar{\alpha_{i_{j}}}$. In particular, we see that the root $\alpha$ is simple for the simple system $s_{i_{j-1}} \cdots s_{i_{2}} s_{i_{1}} \Pi$, hence the parity of $\alpha$ is determined by  the parity of vertex $i_{j}$ in the coloured graph $s_{i_{j-1}} \cdots s_{i_{2}} s_{i_{1}} \Gamma_{col}$. 

For each vertex $i\in \Gamma_{col}$ and $p\in \Z_{2}$ define a simple object $\SSS_{i}^{p} \in \SRep (\Gamma_{col} , \Om)$ by setting
$$\SSS_{i}^{p} (j) = \begin{cases}
\mathbb{K}^{1|0} &\text{ if } i=j \text{ and }  p=0\\
\mathbb{K}^{0|1} &\text{ if } i=j \text{ and } p= 1 \\
0 &\text{ otherwise.} \\
\end{cases}$$
In particular, $p(\SSS_{i}^{p})= p$.

Now define an object $X_{\alpha} = \widetilde{S_{i_{1}}^{-}} \widetilde{S_{i_{2}}^{-}} \cdots \widetilde{S_{i_{j-1}}^{-}} (\SSS_{i_{j}}^{p_{j}})$, where $p_{j}=\text{parity of vertex } i_{j}$ in the graph $s_{i_{j-1}} \cdots s_{i_{2}} s_{i_{1}} \Gamma_{col}$.

Recall that for an indecomposable object $X$, we define the parity of $X$, by $p(X) = \sum_{i} p(X(i))$.
\begin{thm}\label{t:main}
Let $\alpha$ be a positive root for $\Pi$, and $\Gamma_{col}$ denote the coloured Dynkin graph determined by $\Pi$. Then the object $X_{\alpha}$ defined above is indecomposable, has dimension vector $\alpha$, and has the same parity as $\alpha$. Moreover, any indecomposable representation of dimension $\alpha$ is isomorphic to $X_{\alpha}$. 
\end{thm}
\begin{proof}
The fact that $\alpha$ has the same parity as $X_{\alpha}$ follows from the fact that the root $\alpha$ is simple for the system $s_{i_{j-1}} \cdots s_{i_{2}} s_{i_{1}} \Pi$, and the definition of $X_{\alpha}$.
The remaining statements follow by applying the forgetful functor $\SRep (\Gamma, \Om) \to \Rep (\Gamma , \Om)$ and applying the results from \cite{bgp} (see Theorem~\ref{t:bgp}).
\end{proof}

\begin{example}
Consider the case of $A(2,2)$, with $\Pi = \{ \alpha_{1} = \epsilon_{1} - \epsilon_{2} , \alpha_{2} = \epsilon_{2} - \delta_{1} , \alpha_{3} = \delta_{1} - \delta_{2}  \}$. Let $\Om$ be the orientation $\bigcirc \longleftarrow \bigotimes \longleftarrow \bigcirc$. The corresponding adapted expression for $w_{0}$ is $w_{0} = s_{1}s_{2}s_{3}s_{1}s_{2}s_{1}$. Then the representations $X_{\alpha}$ are given as follows:

$X_{\epsilon_{1} - \epsilon_{2}} = \SSS_{1}^{0}= \stackrel{1|0}{\bigcirc} \longleftarrow \stackrel{0|0}{\bigotimes} \longleftarrow \stackrel{0|0}{\bigcirc} \\$

$X_{\epsilon_{1}-\delta_{1}} = \widetilde{S_{1}} (\stackrel{0|0} \bigcirc \longrightarrow \stackrel{0|1}{\bigotimes} \longleftarrow \stackrel{0|0}{\bigcirc}) = \stackrel{1|0}{\bigcirc} \longleftarrow \stackrel{0|1}{\bigotimes} \longleftarrow \stackrel{0|0}{\bigcirc} \\$

$X_{\epsilon_{1}-\delta_{2}} = \widetilde{S_{1}} \widetilde{S_{2}} (\stackrel{0|0}{\bigotimes} \longleftarrow \stackrel{0|0}{\bigotimes} \longrightarrow \stackrel{0|1}{\bigotimes}) = \stackrel{1|0}{\bigcirc} \longleftarrow \stackrel{0|1}{\bigotimes} \longleftarrow \stackrel{1|0}{\bigcirc} \\$

$X_{\epsilon_{2} -\delta_{1}} =\widetilde{S_{1}} \widetilde{S_{2}} \widetilde{S_{3}} (\stackrel{0|1}{\bigotimes} \longleftarrow \stackrel{0|0}{\bigcirc} \longleftarrow \stackrel{0|0}{\bigotimes}) = \stackrel{0|0}{\bigcirc} \longleftarrow \stackrel{0|1}{\bigotimes} \longleftarrow \stackrel{0|0}{\bigcirc} \\$

$X_{\epsilon_{2} - \delta_{2}} =  \widetilde{S_{1}} \widetilde{S_{2}} \widetilde{S_{3}} \widetilde{S_{1}} ( \stackrel{0|0}{\bigotimes} \longrightarrow \stackrel{0|1}{\bigotimes} \longleftarrow \stackrel{0|0}{\bigotimes}) = \stackrel{0|0}{\bigcirc} \longleftarrow \stackrel{0|1}{\bigotimes} \longleftarrow \stackrel{1|0}{\bigcirc} \\$

$X_{\delta_{1} - \delta_{2}} = \widetilde{S_{1}} \widetilde{S_{2}} \widetilde{S_{3}} \widetilde{S_{1}} \widetilde{S_{2}} ( \stackrel{1|0}{\bigcirc} \longleftarrow \stackrel{0|0}{\bigotimes} \longrightarrow \stackrel{0|0}{\bigcirc}) = \stackrel{0|0}{\bigcirc} \longleftarrow \stackrel{0|0}{\bigotimes} \longleftarrow \stackrel{1|0}{\bigcirc} \\
$

Here $n|m$ denotes the super-vector space with an $n$-dimensional vector space in degree 0, and an $m$-dimensional vector space in degree 1. All non-zero maps are given by parity change.
\end{example}

To obtain a construction giving all the roots, we extend this construction by setting $X_{-\alpha} = T X_{\alpha} \in \D^{b} (\SRep (\Gamma , \Om)) /T^{2}$, where $T$ denotes the translation functor.

\begin{cor}
Let $\K$ be the Grothendieck group of $\SRep (\Gamma_{col} , \Om)$. Then $\K$ is isomorphic to the root lattice of $A(n,m)$ and the $\Z_{2}$ grading on roots is given by the parity of $X_{\alpha}$ in Theorem~\ref{t:main}. In particular, the set of indecomposables, coloured by the parity of the object $X_{\alpha}$, gives the root system.
\end{cor}

\begin{cor}
The bijection $R \to \Gammahat$ given by $\alpha \mapsto [X_{\alpha}]$ gives a combinatorial realisation of the root system $A(n,m)$ in terms of the Auslander-Reiten quiver $\Gammahat$, where we colour vertices of $\Gammahat$ according to the parity of $X_{\alpha}$. 
\end{cor}

\begin{example}
Again consider the case of $A(2,2)$. The bijection $R \to \Gammahat$ is given below.
\end{example}

$$
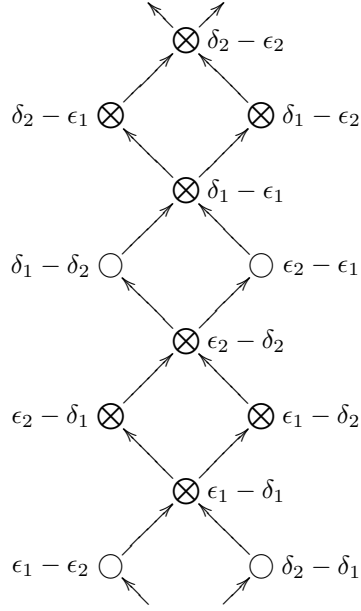
\xy
(0,0)*{\bigcirc}="A"; (10,10)*{\bigotimes}="B"; (0,20)*{\bigotimes}="C"; (10, 30)*{\bigotimes}="D"; (0,40)*{\bigcirc}="E"; (10,50)*{\bigotimes}="F"; (0,60)*{\bigotimes}="G"; (10,70)*{\bigotimes}="H"; (5,75)*{}="I"; (5,-5)*{}="P"; 
(20,0)*{\bigcirc}="J"; (20,20)*{\bigotimes}="K"; (20,40)*{\bigcirc}="L"; (20,60)*{\bigotimes}="M"; (15,75)*{}="N"; (15,-5)*{}="Q";
(-8,0)*{\epsilon_{1}-\epsilon_{2}}; (-8, 20)*{\epsilon_{2} - \delta_{1}}; (-8, 40)*{\delta_{1} -\delta_{2}}; (-8, 60)*{\delta_{2} - \epsilon_{1}};
(18, 10)*{\epsilon_{1} - \delta_{1}}; (18, 30)*{\epsilon_{2} - \delta_{2}}; (18, 50)*{\delta_{1} - \epsilon_{1}}; (18, 70)*{\delta_{2} - \epsilon_{2}}; 
(28, 0)*{\delta_{2} - \delta_{1}}; (28, 20)*{\epsilon_{1} - \delta_{2}}; (28, 40)*{\epsilon_{2} - \epsilon_{1}}; (28, 60)*{\delta_{1} - \epsilon_{2}};

\ar "B"; "A" ;
\ar "C"; "B";
\ar "D"; "C";
\ar "E"; "D";
\ar "F"; "E";
\ar "G"; "F";
\ar "H"; "G";
\ar "I"; "H";
\ar "A"; "P";

\ar "B"; "J";
\ar "K"; "B";
\ar "D"; "K";
\ar "D" ; "L"
\ar "F"; "L";
\ar "M"; "F";
\ar "H"; "M";
\ar "N"; "H";
\ar "J"; "Q";

\endxy
$$

Recall that $\Gammahat$ is periodic, so the arrows leaving the top row are the same as the incoming arrows of the bottom row.

\section{Graded Path Algebra of a Coloured Quiver}

In this section we show that the path algebra of a $\Z_{2}$-coloured quiver $\overrightarrow{Q}$ has a natural $\Z_{2}$ grading, and that the category of $\Z_{2}$-graded modules over this algebra is equivalent to the category of super-representations of $\overrightarrow{Q}$. This gives a generalization of the non-graded case, where modules over the path algebra are equivalent to quiver representations. Moreover, using this construction, we can define a $\Z_{2}$-graded preprojective algebra. The study of this preprojective algebra and its relation to the representation theory of the corresponding Lie superalgebra is the subject of ongoing research.

We begin by recalling the definition of the path algebra of a quiver $\overrightarrow{Q}$.

For any quiver $\overrightarrow{Q}$ let $\PQ$ be the following algebra. As an algebra it is generated by elements $\{ e \}_{e \in Q_{1}} \cup \{ v_{i} \}_{i\in Q_{0}}$. Here the elements $v_{i}$ are thought of as ``paths of length 0 from $i$ to $i$". Viewing a path as a sequence of edges, the multiplication of basis elements is given by concatenation of paths. 

\begin{defi} The algebra $\PQ$ defined above is called the {\em path algebra} of $\overrightarrow{Q}$. It is an associative algebra with unit given by $1 = \sum_{i \in Q_{0}} v_{i}$.
\end{defi} 

The algebra $\PQ$ is graded by path length and by the source and target of the path. This gives a decomposition 
\begin{equation}\label{e:decomp}
\PQ = \bigoplus_{i,j \in Q_{0} ; k \in \N} P_{i,j;k}
\end{equation} 
where $P_{i,j;k}$ is the space spanned by paths of length $k$ from $i$ to $j$. (Here an edge has length 1, and the idempotent corresponding to a vertex has length 0.)

We now define a $\Z_{2}$ grading on $\PQ$ in the case that $\overrightarrow{Q}$ is a coloured quiver. 
\begin{itemize}
\item For each vertex $i \in Q_{0}$ set $p(v_{i}) = 0$. 
\item For each edge $e\in Q_{1}$ set $p(e) = p(s(e)) + p(t(e)) \in \Z_{2}$. 
\item Viewing a path as a sequence of edges, we extend this grading to all paths, by setting $p(e_{i_{n}} e_{i_{n-1}} \cdots e_{i_{1}}) = \sum p(e_{i_{k}})$. 
\item We set $\PQ_{i} = \{ x \in \PQ | p(x) = i \}$. 
\end{itemize}
Since multiplication of paths is given by concatenation, and $p(v_{i}) = 0$, we see that this grading is well-defined and $m: \PQ_{i} \otimes \PQ_{j} \to \PQ_{i+j}$.

\begin{defi} A super-module $M$ over $\PQ$ is a $\Z_{2}$-graded vector space $M=M_{0} \oplus M_{1}$, such that for any edge $e$, the action $e:M \to M$ has degree $p(e)$.\\
A morphism of modules $M \to N$, is a linear map $\Phi : M \to N$, which is compatible with the action of $\PQ$. (i.e. For any path $p \in \PQ$, $\Phi (p.x) = p. \Phi(x)$.) \\
Denote the abelian category of super-modules over $\PQ$ by $SMod (\PQ)$.
\end{defi}

In the case of uncoloured quivers and non-graded modules, there is an equivalence between the category of modules over the path algebra $\PQ$ and the category of representations of $\overrightarrow{Q}$. This equivalence extends to this setting.\\

Define a functor $F: SRep (\overrightarrow{Q}) \to SMod (\PQ)$  as usual, by setting $F(X) = \oplus_{i \in Q_{0}} X(i)$ with the action of the edge $e$ given by the map $x_{e}$. Note that the degree of the map $x_{e}$ is the same as the parity of $e \in \PQ$, so this is well-defined. \\

Define a functor $G: SMod (\PQ) \to SRep (\overrightarrow{Q})$ as usual, by setting  $G(M) = X_{M}$ where $X_{M} (i) = v_{i} M$ and $x_{e}$ is given by the action of $e : v_{s(e)}M \to v_{t(e)}M$. Again, since $p(e) = p(s(e)) + p(t(e))$ we see that  $x_{e}$ has the correct degree.\\

\begin{prop}\label{p:equiv2}
The functors $F,G$ are inverse and provide an equivalence of categories between $SRep (\overrightarrow{Q})$ and $SMod (\PQ)$.
\end{prop}

\begin{proof}
This follows easily from the definitions, and is exactly the same as the classical case. (See \cite{crawley-boevey}.)
\end{proof}

We now consider the preprojective algebra of a coloured graph $\Gamma_{col}$.

The preprojective algebra of a quiver $\overrightarrow{Q}$ is defined as follows:
Consider the double quiver $\overline{Q}$ which has the same vertex set as $\overrightarrow{Q}$ but for every arrow $e:i\to j$ there is an arrow $\overline{e} :j \to i$.
Choose a function $\epsilon : \overline{Q}_{1} \to \{ \pm 1 \}$ so that $\epsilon (e) + \epsilon ( \overline{e} ) = 0$. For each vertex $i\in \overrightarrow{Q}$ define $\theta_{i} \in P_{i,i;2}$ by
\begin{equation}\label{e:mesh}
\theta_{i} =\sum_{s(e)=i}  \epsilon (e)  \overline{e} e \in P_{i,i;2}
\end{equation}

\begin{defi} The {\em preprojective algebra}  $\Pi (\Gamma)$ of a graph $\Gamma$ is defined as $P(\overline{\Gamma}) / J$ where $J$ is the ideal generated by the $\theta_{i}$'s. The ideal $J$ is called the ``mesh" ideal.
\end{defi}
Note that this algebra is independent of the choice of $\epsilon$ and depends only on the underlying graph $\Gamma$, not on the orientation $\Om$. (See \cite{lusztig2} for details.)

If we consider the case of a $\Z_{2}$-coloured graph $\Gamma_{col}$, then from the construction above we see that $P(\overline{\Gamma_{col}})$ has a natural $\Z_{2}$-grading, and the mesh relation (Equation~\ref{e:mesh}) is homogeneous of degree 0. To see this, note that $\theta_{i}$ is a sum of paths that start and end at $i$, and hence has degree 0 in $P(\overline{\Gamma})$. Hence $J$ is a homogeneous ideal of degree 0.

This implies that $\Pi (\Gamma_{col}) =P(\overline{\Gamma_{col}}) / J $ inherits a $\Z_{2}$ grading from $P(\overline{\Gamma})$.

\begin{remark} Note that our constructions do not seem to easily extend to non-A-type Lie superalgebras. This is related to the issue that for non-A-type Lie superalgebras, not only the colouring of the Dynkin graph depends on the choice of simple roots, but the underlying graph does as well. (For example, the D series Lie superalgebras can have both simply-laced and non-simply-laced diagrams, depending on the simple roots chosen to construct the Dynkin graph.)
\end{remark}

\bibliographystyle{amsalpha}

\end{document}